\documentclass[11pt]{amsart}
\usepackage{amssymb,amsmath,cite}
\usepackage{amsfonts,cite,amsthm}
\renewcommand{\le}{\leqslant}
\renewcommand{\ge}{\geqslant}

\newtheorem{theorem}{Theorem}
\newtheorem{lemma}{Lemma}
\newtheorem{propos}{Proposition}

\theoremstyle{definition}
\newtheorem{remark}{Remark}

\begin{document}

\title{On H\"{o}lder and Minkowski Type Inequalities}
\author{Petr Chunaev}
\thanks{The research of P. Chunaev was partially supported by MTM 2011-27637 grant.}
\author{Ljiljanka Kvesi\'{c}}
\author{Josip Pe\v{c}ari\'{c}}
\keywords{H\"{o}lder's inequality, Minkowski's inequality}
\subjclass{26D15}

\begin{abstract}
We obtain inequalities of H\"{o}lder and Minkowski type with weights generalizing both the case of weights with alternating signs and the classical case of non-negative weights.
\end{abstract}

\maketitle

\section{Introduction}

Recently Chunaev \cite{Chunaev} obtained H\"{o}lder and Minkowski type inequalities with alternating signs. His results are a supplement to Jensen type inequalities with alternating signs obtained earlier by Szeg\H{o} \cite{Szego},
Bellman \cite{Bellman1953, Bellman1959}, Brunk \cite{Brunk}, and others (see \cite{Weinberger,Steffensen,Olkin,Wright,Pecaric3,Pecaric2},\cite[\S5.38]{PecaricBook} and also Remark~\ref{rem1}).

In this paper, we intend to give inequalities of H\"{o}lder and Minkowski type with more general weights, including both the case of weights with alternating signs and the classical case of non-negative weights (see, for instance, \cite[\S 4.2]{PecaricBook} and \cite{Chunaev,Zhuang}). Namely, weights $p_k$, $k=1,\ldots,n$, satisfying the property
$$
P_k\ge 0, \qquad \text{where}\qquad P_k:=\sum_{m=1}^k p_m, \qquad k=1,\ldots,n,
$$
are considered. We follow proofs in \cite{Chunaev} with several changes in order to obtain our results.

In what follows, we denote non-negative sequences of real numbers in bold print, for example, $\mathbf{a}=\{a_k\}_{k=1}^{n}$ or $\mathbf{b}=\{b_k\}_{k=1}^{n}$, where $n$ is a positive integer or infinity. Expressions like $\mathbf{a}\equiv 1$ mean that all elements of $\mathbf{a}$ equal $1$.
In proofs we use several well-known inequalities for $\alpha,\beta\ge0$ and $p\ge 1$:
\begin{eqnarray}
\label{Jensen}    &(\alpha+\beta)^p\le \; 2^{p-1} (\alpha^p+\beta^p) &(\text{Jensen's inequality}); \\
\label{Young}     &\alpha \beta\le \frac{\alpha^p}{p}+\frac{\beta^q}{q},\quad  \frac{1}{p}+\frac{1}{q}=1 &(\text{Young's inequality});\\
\label{ineq+-} &p\,\beta^{p-1}\le \frac{\alpha^p-\beta^p}{\alpha-\beta}\le p\,\alpha^{p-1},\quad \alpha> \beta& (\text{See  \cite[Th.~41]{HLP}}).
 \end{eqnarray}

\section{H\"{o}lder type inequalities}
\label{Sec}
In this section, we show that there is no a direct analog of  H\"{o}lder's inequality in the case of our weights, but one of reverse H\"{o}lder's inequality exists. Note that reverse H\"{o}lder's inequalities for non-negative weights are well studied (see \cite{Zhuang}).

\begin{theorem} Let $\mathbf{a}$ and $\mathbf{b}$ be non-increasing and such that
$$
0<a\le a_k\le A <\infty, \qquad 0<b\le b_k\le B <\infty,\qquad k=1,\ldots,n.
$$
If, moreover, $P_k\ge 0$, $k=1,\ldots,n$, and $p,q>1$, $1/p+1/q=1$, then
\begin{equation}
0\le \frac{\left(\sum_{k=1}^n p_k a_k^q\right)^{1/q}\left(\sum_{k=1}^n p_k b_k^p\right)^{1/p}}{\sum_{k=1}^n p_k a_k b_k}
\le \left(pA/a\right)^{1/p}\left(qB/b\right)^{1/q}.
\label{BC}
\end{equation}
The left hand side of~$(\ref{BC})$ should be read as there exist no positive constant, depending on $a,A,b,B,p$ or $q$, which bounds the fraction in $(\ref{BC})$ from below.
\label{Holder1}
\end{theorem}
Before the proof of Theorem~\ref{Holder1}, we establish the following fact.
\begin{lemma}
\label{lemmaM}
Let $\mathbf{a}$ be non-increasing, $\mathbf{b}$ be non-decreasing and such that ${b_k\le B}$ for $k=1,\ldots,n$. If, moreover, $P_k\ge 0$ for $k=1,\ldots,n$, then
$$
\sum_{k=1}^n p_k a_k b_k\le B \sum_{k=1}^n p_k a_k.
$$
\end{lemma}
\begin{proof} Applying the Abel transformation, we have
$$
B \sum_{k=1}^n p_k a_k-\sum_{k=1}^n p_k a_k b_k
=\sum_{k=1}^{n-1}P_k (a_k(B-b_k)-a_{k+1}(B-b_{k+1}))+P_n a_n(B-b_n),
$$
where the latter expression is non-negative since the sequences $\mathbf{a}$ and $\{B-b_k\}$ are non-increasing,
and $P_k\ge 0$. The equality holds for example if $\mathbf{b}\equiv B$.
\end{proof}
\begin{proof} We denote the fraction in (\ref{BC}) by $F_{\text{H}}$.
Applying the Abel transformation to the numerator and the denominator of  $F_{\text{H}}$ easily yields ${F_{\text{H}}\ge 0}$. But we prove even more, namely, that there exist no positive constants bounding $F_{\text{H}}$ from below. Following~\cite{Chunaev}, let $p_k=(-1)^{k+1}$, $k=1,\ldots,n$, where $n$ is even, and $\mathbf{a}=\{a_1,a_1,a_3,a_3,\ldots,a_n,a_n,\ldots\}$ be positive and non-decreasing. The sequence $\mathbf{b}$ is arbitrary except such that ${b_{2k-1}-b_{2k}=0}$ for all $k=1,\ldots,n/2$. It follows that
$$
F_{\text{H}}
=\frac{0\cdot\left(\sum_{k=1}^n (-1)^{k+1}b_k^p\right)^{1/p}}{\sum_{k=1}^{n/2} a_{2k-1}( b_{2k-1}-b_{2k})}=0.
$$
Thus $F_{\text{H}}$ cannot be bounded from below by a positive absolute constant or a constant depending on $p$, $q$, maximum or minimum elements of $\mathbf{a}$ and $\mathbf{b}$.

Now we prove the right hand side of (\ref{BC}). Here $N_{\text{H}}$ denotes the numerator of~$F_{\text{H}}$. First we apply the Abel transformation:
$$
N_{\text{H}}=
\left(\sum_{k=1}^{n-1} P_k(a_k^q-a_{k+1}^q)+ P_n a_n^q \right)^{1/q}
\left(\sum_{k=1}^{n-1} P_k(b_k^p-b_{k+1}^p)+ P_n b_n^p \right)^{1/p}.
$$
By the right hand side of (\ref{ineq+-}) and the Abel transformation
$$
N_{\text{H}}\le
\frac{(qA^{q-1})^{1/q}(pB^{p-1})^{1/p}}{C^{1/q}D^{1/p}}
\left(\sum_{k=1}^{n} C p_k a_k\right)^{1/q}
\left(\sum_{k=1}^{n} D p_k b_k\right)^{1/p},
$$
where $C$ and $D$ are arbitrary positive constants. Therefore, (\ref{Young}) after several simplifications gives
$$
N_{\text{H}} \le
\frac{(pA)^{1/p}(qB)^{1/q}}{C^{1/q}D^{1/p}}\left(\sum_{k=1}^n p_k\left(\frac{C}{qb_k}+\frac{D}{pa_k}\right)a_k b_k \right).
$$
In the latter expression, $\{C/(qb_k)+D/(pa_k)\}$ is non-decreasing and $\{a_k b_k\}$ is non-increasing, because $\mathbf{a}$ and $\mathbf{b}$ are non-increasing. Hence by Lemma~\ref{lemmaM}
\begin{equation*}
\begin{split}
N_{\text{H}}&\le \frac{(pA)^{1/p}(qB)^{1/q}}{C^{1/q}D^{1/p}}
\max_k\left\{\frac{C}{qb_k}+\frac{D}{pa_k}\right\}\sum_{k=1}^n p_k a_k b_k\\
            &\le
(pA)^{1/p}(qB)^{1/q}
\left(\frac{1}{qb}\left(\frac{C}{D}\right)^{1/p}+\frac{1}{pa}\left(\frac{D}{C}\right)^{1/q}\right)\sum_{k=1}^n p_k a_k b_k.
\end{split}
\end{equation*}
It is easily seen that in order to get the smallest constant in the latter inequality, we must choose
$C/D=b/a$. It gives the right hand side of (\ref{BC}). Note that the constant there belongs to $(1;\infty)$.
\end{proof}
\begin{remark}
From Theorem~\ref{Holder1}, it is seen that the constant in the right hand side of (\ref{BC}) tends to infinity as $a\to 0$ or $b\to 0$ (note that this constant is better than in \cite{Chunaev}).
Now we give an example of sequences confirming this \cite{Chunaev}.  In Theorem~\ref{Holder1} we suppose that $p_k=(-1)^{k+1}$, $n=2m+1$, $\mathbf{a}\equiv 1$ and $b=b_{2m+1}\to 0$ in $\mathbf{b}$. It gives
\begin{equation*}
\begin{split}
F_{\text{H}}&=\frac{\left(\sum_{k=1}^{2m+1} (-1)^{k+1}a_k^q\right)^{1/q}\left(\sum_{k=1}^{2m+1} (-1)^{k+1}b_k^p\right)^{1/p}}
{\sum_{k=1}^{2m+1} (-1)^{k+1}a_k b_k}\\
            &= \frac{\left(\sum_{k=1}^{2m} (-1)^{k+1}b_k^p\right)^{1/p}}
{\sum_{k=1}^{2m} (-1)^{k+1} b_k}.
\end{split}
\end{equation*}
From the left hand side of (\ref{ineq+-}) we deduce
$$
F_{\text{H}}=\frac{\left(\sum_{k=1}^{m} (b_{2k-1}^p-b_{2k}^p)\right)^{1/p}}
{\sum_{k=1}^{m} (b_{2k-1}-b_{2k})}\ge
p^{1/p}\left(\frac{b_{2m}}{\sum_{k=1}^{m}(b_{2k-1}-b_{2k})}\right)^{1-1/p},
$$
where $1-1/p>0$. Therefore, for a fixed positive $b_{2m}$  the sum in the denominator can be made sufficiently small by an appropriate choice of~$\mathbf{b}$. Consequently, $F_{\text{H}}$ can be arbitrarily large. The same is for $a=a_{2m+1}\to0$.
\end{remark}
It is clear that if $p=q=2$ then the constant in the right hand side of (\ref{BC}) equals $2\sqrt{AB(ab)^{-1}}\ge 2$. Now we give a more precise constant belonging to $[1;\infty)$ for the case when $\mathbf{a}$ and $\mathbf{b}$ satisfy several additional conditions.
\begin{propos}
\label{Cauchy1}
Let $\mathbf{a}$ and $\mathbf{b}$ be non-increasing and such that the sequence $\{a_k/b_k\}$ is monotone and $0<m\le a_k/b_k\le M<\infty$. If, moreover, $P_k\ge 0$ for $k=1,\ldots,n$, then
\begin{equation}
\label{CBS2}
0\le \frac{\sum_{k=1}^n p_k a_k^2 \sum_{k=1}^n p_k b_k^2}{\left(\sum_{k=1}^n p_k a_kb_k\right)^2}\le \tfrac{1}{4}\left(\tfrac{m}{M}+\tfrac{M}{m}\right)^2.
\end{equation}
The left hand side of~$(\ref{CBS2})$ should be read as there exists no positive constant, depending on $m$ and $M$, which bounds the fraction in $(\ref{CBS2})$ from below.
\end{propos}
\begin{proof} The left hand side inequality follows by the same method as in the proof of Theorem~\ref{Holder1}.
To prove the right hand side we denote  the numerator of the fraction in (\ref{CBS2}) by $N_{\text{C}}$.
First we suppose  $\{a_k/b_k\}$ to be non-decreasing, so $1\le a_k/(mb_k)\le M/m$.
Applying (\ref{Young}) with $p=q=2$ yields
\begin{equation*}
\begin{split}
N_{\text{C}}\le \frac{1}{4m^2}\left(\sum_{k=1}^np_k({a_k}^2+{(mb_k)}^2)\right)^2
            = \frac{1}{4}\left(\sum_{k=1}^n p_k\left(\frac{{a_k}}{mb_k}+\frac{mb_k}{{a_k}}\right){a_k} b_k\right)^2.
\end{split}
\end{equation*}
In the latter expression, the sequence $\{c_k+1/c_k\}$, where $c_k=a_k/(mb_k)$, is non-decreasing. Indeed, $\{c_k\}$ is  non-decreasing and  moreover $c_1\ge 1$. Since $f(x)=x+1/x$ is convex for $x\in(0;\infty)$ and has a minimum at $x=1$, the sequence $\{f(c_k)\}$ is non-decreasing. From this by Lemma~\ref{lemmaM}
\begin{equation*}
\begin{split}
N_{\text{C}}\le \frac{1}{4}\left(\max_k\left\{f(c_k)\right\}\right)^2
\left(\sum_{k=1}^n(-1)^{k+1}a_k b_k\right)^2,
\end{split}
\end{equation*}
where $\max_k\left\{f(c_k)\right\}=m/M+M/m$. Supposing  $\{a_k/b_k\}$ to be non-increasing and taking into account that $m/M\le a_k/(Mb_k)\le 1$, we obtain the right hand side of (\ref{CBS2}) by the same technique.

It is easily seen that equality in (\ref{CBS2}) holds for example if $\mathbf{a}\equiv \mathbf{b}$. The fact that the constant in the right hand side of (\ref{CBS2}) belongs to $[1;\infty)$ is obvious.
\end{proof}
From the well-known weighted inequality of arithmetic and geometric means (see for example \cite[Ch. 2]{HLP}) supposing $a_m\ge 0$ and $v_m>0$, we have
\begin{equation}
\label{Young-mult}
\prod_{m=1}^M a_m \le \sum_{m=1}^M v_m a_m^{1/v_m}, \qquad \sum_{m=1}^M v_m=1.
\end{equation}
This is a multivariable version of Young's inequality~(\ref{Young}). From this we obtain a multivariable version of Theorem~\ref{Holder1} (but with less precise constant).
\begin{propos}
Let $\mathbf{x}_m:=\{x_{m,k}\}_{k=1}^n$ be non-increasing sequences such that
$0<a_{m}\le x_{m,k}\le A_m <\infty$, where $m=1,\ldots,M$. If, moreover, $P_k\ge 0$, ${k=1,\ldots,n}$,
and $w_k>0$, $k=1,\ldots,n$, are such that $\sum_{m=1}^M w_m=1$, then
\begin{equation}
0\le \frac{\prod_{m=1}^M \left(\sum_{k=1}^n p_k x_{m,k}^{1/w_m}\right)^{w_m}}{\sum_{k=1}^n p_k \prod_{m=1}^M x_{m,k}}
\le \sum_{m=1}^M A_m^{1/w_m-1} \prod_{j=1, j\neq m}^M\frac{A_j^{1/w_j-1}}{w_j a_j}.
\label{BC-mult}
\end{equation}
The left hand side of~$(\ref{BC-mult})$ should be read as there exists no positive constant, depending on $a_m,A_m$ and $w_m$, which bounds the fraction in $(\ref{BC-mult})$ from below.
\label{Holder-mult}
\end{propos}
\begin{proof} Set $F_{\text{H}}$ is the fraction in (\ref{BC-mult}). Non-existence of a positive constant bounding $F_{\text{H}}$ from below, follows from Theorem~\ref{Holder1}.
To prove the right hand side we denote the numerator of~$F_{\text{H}}$ by $N_{\text{H}}$. By the Abel transformation
$$
N_{\text{H}}=
\prod_{m=1}^M \left(\sum_{k=1}^{n-1} P_k (x_{m,k}^{1/w_m}-x_{m,k+1}^{1/w_m})+P_n x_{m,n}^{1/w_m}\right)^{w_m}.
$$
The right hand side of (\ref{ineq+-}) and the Abel transformation yields
$$
N_{\text{H}}\le
\prod_{m=1}^M \frac{A_m^{1/w_m-1}}{w_m}
\prod_{m=1}^M \left(\sum_{k=1}^{n} p_k x_{m,k}\right)^{w_m}.
$$
Supposing $v_m=w_m$ in (\ref{Young-mult}), we obtain
$$
N_{\text{H}} \le
\prod_{m=1}^M \frac{A_m^{1/w_m-1}}{w_m}
\left(\sum_{k=1}^n p_k\left(\sum_{m=1}^M w_m \prod_{m=1, m\neq k}^M x^{-1}_{m,k}\right)\prod_{m=1}^M x_{m,k} \right),
$$
where it is obvious that $\{\sum_{m=1}^M w_m \prod_{m=1, m\neq k}^M x^{-1}_{m,k}\}_{k=1}^n$ is non-decreasing and $\{\prod_{m=1}^M x_{m,k}\}_{k=1}^n$ is non-increasing. Thus by Lemma~\ref{lemmaM}
\begin{equation*}
\begin{split}
N_{\text{H}}\le \prod_{m=1}^M \frac{A_m^{1/w_m-1}}{w_m}
\max_k\left\{\sum_{m=1}^M w_m \prod_{m=1, j\neq m}^M x^{-1}_{j,k}\right\}\sum_{k=1}^n p_k \prod_{m=1}^M x_{m,k}.
\end{split}
\end{equation*}
Several simplifications give the right hand side of (\ref{BC-mult}).
\end{proof}

\section{Minkowski type inequalities}

In this section we prove precise Minkowski type inequalities with our weights. As we have already mentioned, these generalize both the case of weight with alternating signs and the case of non-negative weights (see \cite{Chunaev}).
\begin{theorem}
\label{Mink}
Let $\mathbf{a}$ and $\mathbf{b}$ be non-negative non-increasing sequences, and $P_k\ge 0$ for $k=1,\ldots,n$. Then
\begin{equation}
\label{Minn}
0\le \frac{\left(\sum_{k=1}^n p_k a_k^p\right)^{1/p}+\left(\sum_{k=1}^n p_k b_k^p\right)^{1/p}}{\left(\sum_{k=1}^n p_k (a_k +b_k)^p\right)^{1/p}} \le 2^{1-1/p}, \quad p\ge 1.
\end{equation}
The constant $2^{1-1/p}$ is best possible. The left hand side of $(\ref{Minn})$ should be read as there exists no positive constant, depending on only $p$, which bounds the fraction in~$(\ref{Minn})$ from below.
\label{Mink0}
\end{theorem}
\begin{proof}
Throughout the proof,  $F_{\text{M}}$ denotes the fraction in (\ref{Minn}).
Applying the Abel transformation for the numerator and the denominator of  $F_{\text{M}}$ easily yields $F_{\text{M}}\ge 0$.
Moreover, there exists no positive constant depending on $p$ only that bounds $F_{\text{M}}$ from below. Indeed \cite{Chunaev}, for each $p>1$ there exists a sequence such that $F_{\text{M}}$  tends to zero. Supposing that $p_k=(-1)^{k+1}$, $n\ge 2$, $\mathbf{a}=\{1,1,0,\ldots,0,\ldots\}$ and $\mathbf{b}=\{b,0,\ldots,0,\ldots\}$ with some $b>0$, from the left hand side of (\ref{ineq+-}) we deduce
$$
F_{\text{M}}=\frac{b}{\left((1+b)^p-1\right)^{1/p}}\le\frac{b}{(pb)^{1/p}}<b^{1-\frac{1}{p}}.
$$
In this way $F_{\text{M}}\to 0$ as $b\to 0$ since $1-1/p>0$ for all $p>1$.

Now we prove the right hand side of (\ref{Minn}). From (\ref{Jensen}) we have
$$
\left(\left(\sum_{k=1}^n p_k a_k^p\right)^{1/p}+\left(\sum_{k=1}^n p_k b_k^p\right)^{1/p}\right)^p
\le 2^{p-1}\left(\sum_{k=1}^n p_k(a_k^p+b_k^p)\right).
$$
Now, before extraction the $p$\;th root, it is enough to show that
\begin{equation}
\label{ineq22}
\sum_{k=1}^n p_k(a_k^p+b_k^p)\le \sum_{k=1}^n p_k(a_k +b_k)^p, \qquad p\ge 1,
\end{equation}
The inequality~(\ref{ineq22}) by the Abel transformation is equivalent to
$$
\sum_{k=1}^n p_k c_k=\sum_{k=1}^{n-1} P_k\left(c_k-c_{k+1}\right)+P_n c_n\ge 0.
$$
where $c_k:=(a_k +b_k)^p-(a_k^p+b_k^p)$. The latter inequality holds since $P_k\ge 0$ for all $k$ and  $c_k\ge c_{k+1}$ for $k=1,\ldots,n-1$. Indeed, for the function $f(x,y)=(x+y)^p-(x^p+y^p)$, where $x\ge0$, $y\ge 0$ and $p\ge 1$, we have $f'_x\ge0$ and $f'_y\ge 0$. Therefore,
$$
f(a_k,y)\ge f(a_{k+1},y),\; f(x,b_{k})\ge f(x,b_{k+1}) \; \Rightarrow\; f(a_k,b_k)\ge f(a_{k+1},b_{k+1}).
$$
This completes the proof of (\ref{ineq22}).

The precision of the constant $2^{1-1/p}$  is come out from the following observation from~\cite{Chunaev}. If $p_k=1$ for all $k$, $\mathbf{a}=\{1,\ldots,1,0,\ldots,0\}$ (first $n$ elements are units) and $\mathbf{b}=\{n^{1/p},0,\ldots,0\}$, then after several simplifications we get
$$
\frac{\left(\sum_{k=1}^n a_k^p\right)^{1/p}+\left(\sum_{k=1}^n b_k^p\right)^{1/p}}{\left(\sum_{k=1}^n (a_k +b_k)^p\right)^{1/p}}=
2\left(1-\tfrac{1}{n}+\left(1+\tfrac{1}{n^{1/p}}\right)^p\right)^{-1/p}=2^{1-1/p}-\varepsilon_n,
$$
where positive $\varepsilon_n\to 0$ as $n\to \infty$.
\end{proof}
\begin{remark}
\label{rem1}
The following Jensen-Steffensen type statement was proved in \cite{Barlow} (see also \cite[\S2.2]{PecaricBook}).

\textit{Let $\textbf{a}$ be a non-increasing positive sequence and $\varphi$ be a function convex on $[a_n;a_1]$ and such that $\varphi(0)=0$. Then the necessary and sufficient condition on weights $p_k$ in order that}
$$
\varphi\left(\sum_{k=1}^n p_k a_k\right)\le \sum_{k=1}^n p_k \varphi(a_k), \qquad P_k=\sum_{m=1}^kp_m,
$$
\textit{is $0\le P_k\le 1$, $k=1,\ldots,n$.}

From this point of view, the sufficient condition $P_k\ge 0$, $k=1,\ldots,n$, in Theorems~\ref{Holder1} and~\ref{Mink} seems to be quite close to the necessary one.
\end{remark}

\section{Further generalizations}

Now we give integral versions of Lemma~\ref{lemmaM} and Theorems~\ref{Holder1} and~\ref{Mink}.
In what follows, we use the notation
\begin{equation}
\label{P(x)}
P(x):=\int_\alpha^x p(t)\,dt, \qquad x\in[\alpha;\beta],
\end{equation}
and suppose that all functions of $x$ are integrable and differentiable on $[\alpha;\beta]$.
\begin{lemma}
\label{lemmaM-int}
For $x\in [\alpha;\beta]$, let $f(x)$ be non-negative and non-increasing, $g(x)$ be non-decreasing and such that ${0\le g(x) \le B}$, and $P(x)\ge 0$. Then
$$
\int_\alpha^\beta f(x)g(x)\,dP(x)\le B \int_\alpha^\beta f(x)\,dP(x).
$$
\end{lemma}
\begin{proof} Applying integration by parts gives
\begin{equation*}
\begin{split}
B &\int_\alpha^\beta f(x)dP(x)-\int_\alpha^\beta f(x)g(x)dP(x)\\
&=\left.P(x)f(x)(B-g(x))\right|_\alpha^\beta-\int_\alpha^\beta P(x)d(f(x)(B-g(x)))\\
&=P(\beta)f(\beta)(B-g(\beta))+\int_\alpha^\beta P(x)\left(f(x)g'(x)-f'(x)(B-g(x))\right)dx \ge 0.
\end{split}
\end{equation*}
Here we took into account that $P(\alpha)=0$; $P(x)$, $f(x)$, $g'(x)$, $B-g(x)$ are non-negative and $f'(x)$ is non-positive for $x\in[\alpha;\beta]$.
It is easily seen that equality holds for example if $g(x)\equiv B$.
\end{proof}
Using Lemma~\ref{lemmaM-int} and intergation by parts instead of the Abel transformation, we obtain the following results by essential repeating proofs of Theorems~\ref{Holder1}~and~\ref{Mink}. We emphasize that $dP(x)$ may be negative here in contrast to the classical case.
\begin{theorem}
\label{Holder-int}
For $x\in[\alpha;\beta]$, let $f(x)$ and $g(x)$ be non-increasing and
$$
0<a\le f(x)\le A <\infty, \qquad 0<b\le g(x)\le B <\infty.
$$
If, moreover, $P(x)\ge 0$, $x\in[\alpha;\beta]$, and $p,q>1$, $1/p+1/q=1$, then
\begin{equation}
0\le \frac{\left(\int_\alpha^\beta f^q(x) dP(x)\right)^{1/q}\left(\int_\alpha^\beta g^p(x)dP(x)\right)^{1/p}}{\int_\alpha^\beta f(x)g(x)dP(x)}
\le \left(pA/a\right)^{1/p}\left(qB/b\right)^{1/q}.
\label{BC-int}
\end{equation}
The left hand side of~$(\ref{BC-int})$ should be read as there exists no positive constant, depending on $a,A,b,B,p$ and $q$, which bounds the fraction in $(\ref{BC-int})$ from below.
\end{theorem}
\begin{theorem}
\label{Mink-int}
For $x\in[\alpha;\beta]$, let $f(x)$ and $g(x)$ be non-negative and non-increasing, and $P(x)\ge 0$. Then
\begin{equation}
\label{Minn-int}
0\le \frac{\left(\int_\alpha^\beta f^p(x)\,dP(x)\right)^{1/p}+\left(\int_\alpha^\beta g^p(x)\,dP(x)\right)^{1/p}}{\left(\int_\alpha^\beta (f(x)+g(x))^p\,dP(x)\right)^{1/p}} \le 2^{1-1/p}, \quad p\ge 1.
\end{equation}
The constant $2^{1-1/p}$ is best possible. The left hand side of $(\ref{Minn-int})$ should be read as there exists no positive constant, depending only on $p$, which  bounds the fraction in~$(\ref{Minn-int})$ from below.
\label{Mink0-int}
\end{theorem}

In conclusion we give several examples concerning Theorems 3 and 4. Let $p(t)=\sin t$ and $x\in [0;\infty)$ in (\ref{P(x)}), then $P(x)=1-\cos x\ge 0$, and thus
$$
0\le \frac{\left(\int_{0}^{\infty} f^q(x)\sin x \;dx\right)^{1/q}\left(\int_{0}^{\infty} g^p(x)\sin x \;dx\right)^{1/p}}{\int_{0}^{\infty} f(x)g(x)\sin x \;dx}
\le \left(pA/a\right)^{1/p}\left(qB/b\right)^{1/q},
$$
$$
0\le \frac{\left(\int_{0}^{\infty} f^p(x)\sin x \;dx\right)^{1/p}+\left(\int_{0}^{\infty} g^p(x)\sin x \;dx\right)^{1/p}}{\left(\int_{0}^{\infty} (f(x)+g(x))^p\sin x \;dx\right)^{1/p}} \le 2^{1-1/p}, \qquad p\ge 1.
$$
Appropriate discretization yields inequalities with alternating signs obtained earlier in~\cite{Chunaev} (the case $p_k=(-1)^{k+1}$ in Theorems~1 and~2).

If $P(x)$ is non-decreasing for $x\in[\alpha;\beta]$ (i.e. $dP(x)$ is non-negative), Theorems~3 and~4 give the classical case of non-negative weights, for which we can put $1$ instead of $0$ in the left hand sides of (\ref{BC-int}) and (\ref{Minn-int}) due to H\"{o}lder's and Minkowski's inequalities.

\bigskip

\bigskip

\begin{flushright}
{\small

\textbf{Petr Chunaev}\\
Centre de Recerca Matem\`{a}tica, Spain\\
E-mail: \textsf{chunayev@mail.ru}

\bigskip

\textbf{Ljiljanka Kvesi\'{c}}\\
University of Mostar, Faculty of Science and Education, Bosnia and Herzegovina\\
E-mail: \textsf{ljkvesic@gmail.com}

\bigskip

\textbf{Josip Pe\v{c}ari\'{c}}\\
University of Zagreb, Faculty of Textile Technology, Croatia\\
E-mail: \textsf{pecaric@mahazu.hazu.hr}
}
\end{flushright}

\end{document}